\documentclass[11pt]{article}
\usepackage{amssymb,amsmath,amsthm}
\usepackage[mathscr]{eucal}
\usepackage[cm]{fullpage}
\usepackage[english]{babel}
\usepackage[latin1]{inputenc}
\usepackage{color}

\usepackage[usenames,dvipsnames]{xcolor} 
\usepackage{graphicx} 
\usepackage{draftcopy} 
\draftcopyName{WATERMARK}{100} 
\usepackage{watermark} 
\definecolor{verylightblue}{rgb}{.855,.89,1.0}
\definecolor{lightbluegray}{rgb}{.788,.847,0.887}
\def\watermarktext{\put(-10,-680){
\rotatebox{
90}{\Huge\bf\color{lightbluegray}
{\scriptsize arXiv (v1.submitted.19022015)}
}}}
\watermark{%
\watermarktext 
} 

\def\dom{\mathop{\mathrm{Dom}}\nolimits}
\def\im{\mathop{\mathrm{Im}}\nolimits}
\def\N{\mathbb N}

\newcommand{\la}{.} 
\def\kunze{\!\Join\!}
\newcommand{\se}{\!\rtimes\!}
\newcommand{\sd}{\!\ltimes\!}

\def\POI{\mathcal{POI}}
\def\DP{\mathcal{DP}}
\def\ODP{\mathcal{ODP}}
\def\PODI{\mathcal{PODI}}

\def\PO{\mathcal{PO}}

\def\I{\mathcal{I}}
\def\pv#1{\ensuremath{{\sf#1}}}

\newtheorem{theorem}{Theorem}[section]

\newtheorem{corollary}[theorem]{Corollary}
\newtheorem{lemma}[theorem]{Lemma}
\newtheorem{conjecture}[theorem]{Conjecture}

\newcommand{\lastpage}{\addresss}

\newcommand{\addresss}{\small \sf  
\noindent{\sc V\'\i tor H. Fernandes}, 
Departamento de Matem\'atica, 
Faculdade de Ci\^encias e Tecnologia, 
Universidade Nova de Lisboa, 
Monte da Caparica, 
2829-516 Caparica, 
Portugal; 
also: 
Centro de \'Algebra da Universidade de Lisboa, 
Av. Prof. Gama Pinto 2, 
1649-003 Lisboa, 
Portugal; 
e-mail: vhf@fct.unl.pt

\medskip

\noindent{\sc Teresa M. Quinteiro}, 
Instituto Superior de Engenharia de Lisboa, 
Rua Conselheiro Em\'\i dio Navarro 1, 
1950-062 Lisboa, 
Portugal; 
also: 
Centro de \'Algebra da Universidade de Lisboa, 
Av. Prof. Gama Pinto 2, 
1649-003 Lisboa, 
Portugal;
e-mail: tmelo@adm.isel.pt 
}

\title{A note on bilateral semidirect product decompositions of some monoids of order-preserving partial permutations}

\author{V\'\i tor H. Fernandes\footnote{This work was developed within the FCT Project PEst-OE/MAT/UI0143/2014 of CAUL, FCUL, and of Departamento de Matem\'atica da Faculdade de Ci\^encias e Tecnologia da Universidade Nova de Lisboa.}~ and 
Teresa M. Quinteiro\footnote{This work was developed within the FCT Project 
PEst-OE/MAT/UI0143/2014 of CAUL, FCUL, and of Instituto Superior de Engenharia de Lisboa.}
}

\begin{document}

\maketitle

\begin{abstract}
In this note we consider the monoid $\PODI_n$ of all monotone partial permutations on 
$\{1,\ldots,n\}$ and its submonoids $\DP_n$, $\POI_n$ and $\ODP_n$ of all partial isometries, 
of all order-preserving partial permutations and of all order-preserving partial isometries, respectively. 
We prove that both the monoids $\POI_n$ and $\ODP_n$ are quotients of bilateral semidirect products of two of their remarkable submonoids, namely of extensive and of co-extensive transformations.    
Moreover, we show that $\PODI_n$ is a quotient of a semidirect product of $\POI_n$ and the group 
$\mathcal{C}_2$ of order two and, analogously, $\DP_n$ is a quotient of a semidirect product of $\ODP_n$ and $\mathcal{C}_2$. 
\end{abstract}

\medskip



\section*{Introduction and preliminaries}

Strongly motivated by automata theoretic ideas, in \cite{Kunze:1983} Kunze studied the notion of  bilateral semidirect product of two semigroups (see \cite{Kunze:1984a,Kunze:1984b} for applications in Automata Theory) and proved in \cite{Kunze:1992} that the full transformation semigroup on a finite set $X$ is a quotient of a bilateral semidirect product of the symmetric group on $X$ and the semigroup of all order-preserving full transformations on $X$, for some linear order on $X$.
Also in \cite{Kunze:1992}, Kunze showed that the semigroup of all order-preserving full transformations on a finite chain is a quotient of a bilateral semidirect product of two of its subsemigroups. These results as well as applications to Formal Languages were also discussed by Kunze in \cite{Kunze:1993}. 
Bilateral semidirect products were also considered by Lavers 
\cite{Lavers:1998} who gave conditions under which
a bilateral semidirect product of two finitely presented monoids is itself
finitely presented, by exhibiting explicit presentations, under some conditions.

In this note we construct bilateral semidirect decompositions, 
i.e. a representation of monoid $S$ as a quotient of a bilateral semidirect product of two proper submonoids of $S$, of certain monoids of partial permutations. 

\smallskip

Denote by $\mathcal{T}(X)$ the semigroup (under composition) of all full transformations
of a set $X$. Let $S$ and $T$ be two semigroups. Let
$$
\begin{array}{llllllll}
\delta:&T&\longrightarrow&\mathcal{T}(S)&&&\\
&u&\longmapsto&\delta_u:&S&\longrightarrow&S\\
&&&&s&\longmapsto&u\la s
\end{array}
$$
be an anti-homomorphism of semigroups (i.e.
$(uv)\la s=u\la(v\la s)$, for $s\in S$ and $u,v\in T$) and let
$$\begin{array}{llllllll}
\varphi:&S&\longrightarrow&\mathcal{T}(T)&&&\\
&s&\longmapsto&\varphi_s:&T&\longrightarrow&T\\
&&&&u&\longmapsto&u^s
\end{array}$$
be a homomorphism of semigroups (i.e.
$u^{sr}={(u^s)}^r$, for $s,r \in S$ and $u \in T$) such that:
\begin{description}
 \item (SPR) $(uv)^s=u^{v\la s}v^s$, for $s \in S$ and $u,v \in T$
(\textit{Sequential Processing Rule}); and
 \item (SCR) $u\la(sr)=(u\la s)(u^s\la r)$, for $s,r \in S$ and $u \in T$
(\textit{Serial Composition Rule}).
\end{description}
Within these conditions, we say that $\delta$ is a \textit{left action} of $T$ on $S$ and that $\varphi$ is a \textit{right action} of $S$ on $T$.

In \cite{Kunze:1983}, Kunze proved that the set $S\times T$ is a semigroup  with respect to the following multiplication:
$$
(s,u)(r,v)=(s(u\la r),u^rv),
$$
for $s,r\in S$ and $u,v\in T$. We denote this semigroup by $S{_\delta\kunze_\varphi}T$ (or, if it is not ambiguous, simply by $S\kunze T$) and call it the \textit{bilateral semidirect product} of $S$ and $T$ associated with $\delta$ and $\varphi$.

If $S$ and $T$ are monoids and the actions $\delta$ and $\varphi$ preserve the identity
(i.e. $1\la s=s$, for $s\in S$, and $u^1=u$, for $u\in T$) and are \textit{monoidal} (i.e.
$u\la 1=1$, for $u\in T$, and $1^s=1$, for $s\in S$), then  $S\kunze T$ is a monoid with identity $(1,1)$.

Here, we will just consider bilateral semidirect products of monoids associated to monoidal actions.

Notice that, if the right action $\varphi$ is a trivial action
(i.e. $(S)\varphi=\{\mathrm{id}_T\}$) then
$S\kunze T=S\se T$ is an usual semidirect product, if the left action $\delta$ is a trivial action (i.e. $(T)\delta=\{\mathrm{id}_S\}$) then $S\kunze T$ coincides with a reverse semidirect product $S\sd T$ and if both actions are trivial then $S\kunze T$ is the usual direct product $S\times T$. Observe also that the bilateral semidirect product is quite different
from the \textit{double} semidirect product by Rhodes and Tilson \cite{Rhodes&Tilson:1989}, wherein the second components multiply always as in the direct product.

\smallskip 

A partial transformation $s$ on the chain $X_n=\{1<2<\cdots<n\}$, $n\in\N$, is said to be
\textit{order-preserving} (respectively, \textit{order-reversing}) if $i\le j$ implies $is\le js$ (respectively, $is\geq js$) for all
$i,j \in\dom(s)$.  Order-preserving and order-reversing partial transformations are also called \textit{monotone}. 

Semigroups of order-preserving transformations have been considered in the literature
since the 1960s. 
In  1962, A\v\i zen\v stat \cite{Aizenstat:1962} and Popova \cite{Popova:1962} exhibited presentations for
$\mathcal{O}_n$, the monoid of all order-preserving full transformations
on $X_n$, and for $\PO_n$, the  monoid of all
order-preserving partial transformations on $X_n$, respectively.
In 1971, Howie \cite{Howie:1971}
studied some combinatorial and algebraic properties of $\mathcal{O}_n$
and, in 1992, together with Gomes \cite{Gomes&Howie:1992}
revisited the monoids $\mathcal{O}_n$ and $\PO_n$. 
More combinatorial properties of these two monoids were presented by Laradji and Umar in \cite{Laradji&Umar:2004,Laradji&Umar:2006}. 
Certain classes of divisors of the monoid $\mathcal{O}_n$ 
were determined in 1995 by Higgins \cite{Higgins:1995} and by Vernitski\u{\i} and Volkov \cite{Vernitskii&Volkov:1995}, in 1997
by Fernandes \cite{Fernandes:1997} and in 2010 by Fernandes and Volkov  \cite{Fernandes&Volkov:2010}.  
In \cite{Kunze:1992} Kunze proved that the monoid $\mathcal{O}_n$ is a quotient of a bilateral semidirect product of its subsemigroups 
${\mathcal{O}}_n^-=\{s\in\mathcal{O}_n\mid is\le i, ~\mbox{for $i\in X_n$}\}$ 
and ${\cal{O}}_n^+=\{s\in\mathcal{O}_n\mid i\le is, ~\mbox{for $i\in X_n$}\}$.  
See also \cite{Kunze:1993,Fernandes&Quinteiro:SF2011,Fernandes&Quinteiro:2011}. 

\smallskip 

The injective counterpart of $\mathcal{O}_n$, 
i.e. the monoid $\POI_n$ of all
injective members of $\PO_n$, has been object of study by the
first author in several papers
\cite{Fernandes:1997,Fernandes:1998,Fernandes:2001,Fernandes:2002,Fernandes:2008}, by Derech in \cite{Derech:1991}, 
by Cowan and Reilly in \cite{Cowan&Reilly:1995}, by Ganyushkin and Mazorchuk in  \cite{Ganyushkin&Mazorchuk:2003}, among other authors. 
Presentations for the monoid $\POI_n$ and for its extension $\PODI_n$,  
the monoid of all monotone partial permutations on $X_n$, 
were given by Fernandes \cite{Fernandes:2001} in 2001 and 
by Fernandes et al.~\cite{Fernandes&Gomes&Jesus:2004} in 2004, respectively.  
The first author together with Delgado
\cite{Delgado&Fernandes:2000,Delgado&Fernandes:2004} have also computed the abelian kernels of the monoids $\POI_n$ and $\PODI_n$. 

\smallskip 

Next, let $s$ be a partial permutation on $X_n$. We say that $s$ is an  
\textit{isometry} if $|is-js|=|i-j|$, for all $i,j \in \dom(s)$. 

The study of semigroups of finite partial isometries was initiated by Al-Kharousi et al.~in \cite{AlKharousi&Kehinde&Umar:2014,AlKharousi&Kehinde&Umar:2014s}. 
The first of these two papers was dedicated to investigate some combinatorial properties of the monoid $\DP_n$ of all partial isometries on $X_n$ and of its submonoid $\ODP_n$ of all order-preserving partial isometries, in particular, their cardinalities. The second one presented the study of some of their algebraic properties, namely Green's structure and ranks.  
On the other hand, in \cite{Fernandes&Quinteiro:2014sub} the authors exhibited presentations for both monoids $\DP_n$ and $\ODP_n$.  
Observe that $\ODP_n$, $\POI_n$, $\DP_n$ and $\PODI_n$ are all inverse submonoids of the symmetric inverse monoid (i.e. the monoid of all partial permutations) 
$\I_n$ on $X_n$ 
(see \cite{AlKharousi&Kehinde&Umar:2014s,Fernandes&Gomes&Jesus:2004}). 
Obviously, $\POI_n\subseteq\PODI_n$ and $\ODP_n=\DP_n\cap\POI_n$ and, 
as observed by Al-Kharousi et al.~\cite{AlKharousi&Kehinde&Umar:2014s}, 
we also have $\DP_n\subseteq\PODI_n$. Moreover, it is easy to check that 
$
\ODP_n=\{s \in\I_n \mid  is-js=i-j, ~\mbox{for $i,j \in\dom(s)$}\}. 
$

\smallskip 

In this paper, in Section \ref{bilateral}, we obtain a bilateral semidirect decomposition of $\POI_n$ in terms of its submonoids 
$\POI_n^+=\{s\in\POI_n\mid i\le is, ~\mbox{for $i\in\dom(s)$}\}$ 
and 
$\POI_n^-=\{s\in\POI_n\mid is\le i, ~\mbox{for $i\in\dom(s)$}\}$ 
of extensive and of co-extensive transformations, respectively. 
A similar decomposition is constructed for the monoid $\ODP_n$ by considering its submonoids 
$\ODP_n^+=\ODP_n\cap\POI_n^+$ 
and 
$\ODP_n^-=\ODP_n\cap\POI_n^-$. 
On the other hand, in Section \ref{unilateral}, we prove that $\PODI_n$ and $\DP_n$ are quotients of semidirect products of the form $\POI_n\se\mathcal{C}_2$ and $\ODP_n\se\mathcal{C}_2$, respectively, where $\mathcal{C}_2$ denotes the group of order two. 
In both sections we extract consequences for pseudovarieties generated by some of these  families of partial permutations monoids. 
 
\smallskip 

Recall that a \textit{pseudovariety} of monoids is a
class of finite monoids closed under formation of finite direct products,
submonoids and homomorphic images. The \textit{semidirect product} 
$\pv{V}\se\pv{W}$ of the pseudovarieties of monoids $\pv{V}$ and $\pv{W}$
is the pseudovariety generated by all monoidal semidirect products $M\se N$,
where $M\in\pv{V}$ and $N\in\pv{W}$. Similarly, we define the \textit{reverse semidirect product}
$\pv{V}\sd\pv{W}$ and the \textit{bilateral semidirect product} $\pv{V}\kunze\pv{W}$
of the pseudovarieties of monoids $\pv{V}$ and $\pv{W}$. 

\smallskip 

Let $\pv{O}$ and $\pv{J}$ be the pseudovarieties of monoids generated by
$\{\mathcal{O}_n\mid n\in\N\}$ and by $\{\mathcal{O}^+_n\mid n\in\N\}$
(or, since $\mathcal{O}_n^+$ and $\mathcal{O}_n^-$ are isomorphic monoids, by $\{\mathcal{O}^-_n\mid n\in\N\}$), respectively.
It is well-known that $\pv{J}$ is the
pseudovariety of $\mathscr{J}$-trivial monoids and that it also is generated by the syntactic monoids of piecewise testable languages (see e.g. \cite{Pin:1986}).
Let $\pv{A}$ be the pseudovariety 
of all aperiodic (i.e. $\mathscr{H}$-trivial) monoids. 
It is easy to show that $\pv{J}\kunze\pv{J}\subseteq \pv{A}$ and,
as an immediate consequence of Kunze's result \cite{Kunze:1992} above mentioned,
we have $\pv{O}\subseteq\pv{J}\kunze\pv{J}$ (see \cite{Fernandes&Quinteiro:SF2011}). 
On the other hand, let $\pv{Ecom}$ be the pseudovariety of all idempotent commuting monoids (recall that a celebrated Theorem of Ash \cite{Ash:1987} states that $\pv{Ecom}$ is generated by all finite inverse monoids) and let $\pv{POI}$ and $\pv{PODI}$ be the pseudovarieties generated by $\{\POI_n\mid n\in\N\}$ and by $\{\PODI_n\mid n\in\N\}$, respectively. Notice that $\pv{POI}\subset\pv{O}\subset\pv{A}$ \cite{Fernandes:1997} 
and that $\pv{J}\cap\pv{Ecom}$ is the pseudovariety generated by $\{\POI^+_n\mid n\in\N\}$ (or, since $\POI_n^+$ and $\POI_n^-$ are isomorphic monoids, 
by $\{\POI^-_n\mid n\in\N\}$) \cite{Higgins:1999}. 
Finally, consider the pseudovariety of monoids $\pv{Ab}_2$ generated by $\mathcal{C}_2$ (a pseudovariety of Abelian groups).  

\smallskip 

For for basic notions on Semigroup Theory, we refer the reader to Howie's book  \cite{Howie:1995}. 

\smallskip 

For simplicity, from now on we consider $n\ge3$. 

\section{On the monoids $\POI_n$ and $\ODP_n$}\label{bilateral}

In this section we show that $\POI_n$ and $\ODP_n$ are homomorphic images of certain bilateral semidirect products of the form $\POI_n^-\kunze \POI_n^+$ and $\ODP_n^-\kunze \ODP_n^+$, respectively. 

\smallskip 

We begin by constructing a bilateral semidirect product $\POI_n^-\kunze \POI_n^+$. 

Let $s,u\in \POI_n\setminus\{1\}$. Define the elements $u\la s,u^s\in\POI_n$ by 
$$
\dom(u\la s)=\dom(us) \quad\text{and}\quad \im(u\la s)=\{1,\ldots,|\dom(us)|\},
$$
$$
\dom(u^s)=\{1,\ldots,|\im(us)|\} \quad\text{and}\quad \im(u^s)=\im(us).
$$

Observe that any element of $\POI_n$ is well defined by its domain and image.

Notice that 
$$
\im(u\la s)=\dom(u^s)
$$
and, clearly,  
$$ 
u\la s\in\POI_n^- \quad\text{and}\quad u^s\in\POI_n^+. 
$$

Define also 
$$
1\la s=s,\quad u^1=u,\quad u\la 1=1,\quad 1^s=1,\quad 1\la 1=1 
\quad\text{and}\quad 1^1=1. 
$$

Consider the following two (well defined) functions: 
$$
\begin{array}{ccccccc}
\delta:&\POI_n^+ &\longrightarrow&\mathcal{T}(\POI_n^-)&&\\
&u&\longmapsto&\delta_u:\POI_n^-&\longrightarrow&\POI_n^-\\
&&&s&\longmapsto&u\la s
\end{array}
$$
and 
$$
\begin{array}{ccccccc}
\varphi:&\POI_n^-&\longrightarrow&\mathcal{T}(\POI_n^+)&&\\
&s&\longmapsto&\varphi_s:\POI_n^+&\longrightarrow&\POI_n^+\\
&&&u&\longmapsto&u^s. 
\end{array}
$$
We have: 

\begin{lemma}\label{morphis}
The above defined functions $\delta$ and $\varphi$ are an anti-homomorphism of monoids and a homomorphism of monoids, respectively. 
\end{lemma}
\begin{proof}
First, notice that $\delta_1$ and $\varphi_1$ are, clearly, the identity maps of $\POI_n^-$ and of $\POI_n^+$, respectively. 

Now, let $u,v\in \POI_n^+$ and $s,r \in \POI_n^-$. Then, we must prove that 
$$
(uv)\la s=u \la (v \la s)\quad\text{and}\quad u^{sr}=(u^s)^r.
$$
It is immediate that, if any of the elements $u$, $v$ or $s$ is the identity, 
then $(uv)\la s=u \la (v \la s)$, and if any of the elements $u$, $s$ or $r$ is the identity, then $u^{sr}=(u^s)^r$. So, let us suppose that none of the elements $u$, $v$, $s$ and $r$ is the identity. 

In order to prove that $(uv)\la s=u \la (v \la s)$, it suffices to show that 
$\dom((uv)\la s)=\dom(u \la (v \la s))$. In fact, we have
$$
\begin{array}{rcl}
\dom ((uv)\la s)&=&\dom((uv)s)\\
&=&\dom(u(vs))\\
&=&(\im(u) \cap \dom (vs))u^{-1}\\
&=&(\im(u)\cap \dom (v \la s))u^{-1}\\
&=&\dom(u(v \la s))\\
&=&\dom(u \la (v \la s)).
\end{array}
$$
 
On the other hand, in order to prove that $u^{sr}=(u^s)^r$, it suffices to show that $\im(u^{sr})=\im((u^s)^r)$: 
$$
\im(u^{sr})=\im(u(sr))=\im((us)r)=(\im (us) \cap \dom r)r=(\im (u^s) \cap \dom r)r=\im (u^s r)=\im((u^s)^r), 
$$
as required. 
\end{proof}

Before proving that $\delta$ and $\varphi$ also verify sequential processing and serial composition rules, we observe that it is easy to check the equality 
\begin{equation}\label{comute}
 (u\la s)u^s=us,
\end{equation}
for all $u\in \POI_n^+$ and $s \in \POI_n^-$. 

\begin{lemma}\label{rules}
Let $s,r \in\POI_n^-$and $u,v \in \POI_n^+$. Then:
\begin{description}
 \item {\rm (SPR)} ${(uv)}^s=u^{v \la s} v^s$;
 \item {\rm (SCR)} $u \la (sr)=(u \la s)(u^s\la r)$.
\end{description}
\end{lemma}

\begin{proof} (SPR) We begin by noticing that if any of the elements $s$, $u$ or $v$ is the identity then the equality ${(uv)}^s=u^{v \la s} v^s$ is obvious. Thus, admit that none of the elements $s$, $u$ or $v$ is the identity. Since $\im(u(v\la s))\subseteq \im(v\la s)$ and taking in account Lemma \ref{morphis}, we obtain 
$$
\begin{array}{rcl}
\dom(u^{v\la s}v^s)&=&(\im(u^{v\la s})\cap \dom(v^s)){(u^{v\la s})}^{-1}\\
&=&(\im(u(v\la s))\cap \im(v\la s)){(u^{v\la s})}^{-1}\\
&=&(\im(u(v\la s))){(u^{v\la s})}^{-1}\\
&=&(\im(u^{v\la s})){(u^{v\la s})}^{-1}\\
&=&\dom(u^{v\la s})\\
&=&\im(u\la(v\la s))\\
&=&\im((uv)\la s)\\
&=&\dom({(uv)}^s).
\end{array}
$$
Then, in particular, $|\im({(uv)}^s)|=|\im(u^{v\la s}v^s)|$. 
Hence, in order to prove that ${(uv)}^s=u^{v \la s} v^s$, 
it suffices to show, for instance, the inclusion $\im(u^{v \la s} v^s)\subseteq \im({(uv)}^s)$.

Let $y\in \im(u^{v\la s}v^s)$. Then there exists 
$x \in \dom(u^{v \la s} v^s)$ such that $y=x(u^{v\la s}v^s)$. It follows that  
$xu^{v\la s} \in \im(u^{v\la s})=\im (u(v\la s))$ 
and so $xu^{v\la s}=a(u(v\la s))$, for some $a \in \dom(u(v\la s))$. 
Thus, by using (\ref{comute}), we have 
$$
y= (xu^{v\la s})v^s=(a(u(v\la s)))v^s= a(u((v\la s)v^s))= a(u(vs))=a((uv)s)\in 
\im((uv)s)=\im ({(uv)}^s), 
$$
which proves the required inclusion. 

(SCR) As for (SPR), if any of the elements $s$, $r$ or $u$ is the identity then the equality $u \la (sr)=(u \la s)(u^s\la r)$ is trivial. Therefore, let us assume 
that none of these elements is the identity. In view of the inclusion 
$\dom(u^s r)\subseteq \dom(u^s)$ and of Lemma \ref{morphis}, we have 
$$
\begin{array}{rcl}
\im((u \la s)(u^s \la r))&=&(\im(u\la s)\cap \dom(u^s\la r))(u^s\la r)\\
&=&(\dom(u^s)\cap \dom(u^sr))(u^s\la r)\\
&=&(\dom(u^sr))(u^s\la r)\\
&=&(\dom(u^s\la r))(u^s\la r)\\
&=&\im(u^s \la r)\\
&=&\dom((u^s)^r)\\
&=&\dom(u^{sr})\\
&=&\im(u\la(sr)).
\end{array}
$$
It follows, in particular, that 
$|\dom((u \la s)(u^s\la r))|=|\dom(u \la (sr))|$ and so it remains to show, 
for instance, that $\dom ((u\la s)(u^s\la r))\subseteq \dom(u\la (sr))$.
Let $x \in \dom ((u\la s)(u^s\la r))$. 
Then $x(u\la s)\in \dom(u^s\la r)=\dom(u^sr)$, whence 
$x(u\la s)u^s \in \dom(r)$ and so, by (\ref{comute}), 
$x(us)=x(u\la s)u^s \in \dom(r)$. 
Thus $x\in \dom ((us)r)=\dom(u(sr))=\dom (u\la (sr))$, as required. 
\end{proof}

Now, by Lemma \ref{morphis} and Lemma \ref{rules}, we can consider the bilateral semidirect product $\POI_n^- \kunze \POI_n^+$ associated with $\delta$ and $\varphi$. Since $\POI_n^-$ and $\POI_n^+$ are monoids and the actions $\delta$ and $\varphi$ preserve the identity and are monoidal, then $\POI_n^- \kunze \POI_n^+$ is also a monoid. Moreover, as we already observed, $\POI_n^-$ and $\POI_n^+$ are (isomorphic) $\mathscr{J}$-trivial monoids and any bilateral semidirect product of $\mathscr{J}$-trivial monoids is an aperiodic semigroup, whence 
$\POI_n^- \kunze \POI_n^+$ is an aperiodic monoid. 
On the other hand, $\POI_n^- \kunze \POI_n^+$ is not regular and is not an idempotent commuting semigroup. For instance, if $e=\binom{1}{1}$ and $f=\binom{12}{12}$, it is routine matter to show that $(e,\emptyset)$ is not regular, $(1,e)$ and $(f,f)$ are idempotents and $(1,e)(f,f)=(e,e)\ne(f,e)=(f,f)(1,e)$. 

Next, consider the following function 
$$
\begin{array}{lccc}
\mu:&\POI_n^-\kunze \POI_n^+&\longrightarrow&\POI_n\\
&(s,u)&\longmapsto&su~. 
\end{array}
$$
Let $(s,u),(r,v)\in \POI_n^-\kunze \POI_n^+$. As $(u\la r) u^r=ur$, we have 
$$
((s,u)(r,v))\mu=(s(u\la r),u^r v)\mu=s(u\la r)u^r v=surv=(s,u)\mu(r,v)\mu,
$$ 
and so $\mu$ is a homomorphism. In addition, given $t\in\POI_n$, we may define elements $s\in\POI_n^-$ and $u\in\POI_n^+$ by 
$$
\dom(s)=\dom(t),\quad \im(s)=\{1,\ldots,|\dom(t)|\}=\dom(u)
\quad\text{and}\quad
\im(u)=\im(t), 
$$ 
and we obtain $t=su=(s,u)\mu$. Hence $\mu$ is onto homomorphism and we have: 

\begin{theorem}
The monoid $\POI_n$ is a homomorphic image of $\POI_n^-\kunze \POI_n^+$.
\end{theorem}

As an immediate consequence of this result and the above observed fact that  
$\POI_n^-\kunze \POI_n^+\not\in\pv{Ecom}$, we have the following property:

\begin{corollary}
$\pv{POI}\subsetneq(\pv{J}\cap\pv{Ecom})\kunze(\pv{J}\cap\pv{Ecom})\not\subseteq\pv{Ecom}$. 
\end{corollary}

\medskip 

Next, we construct a bilateral semidirect product $\ODP_n^-\kunze \ODP_n^+$, 
just by slightly modifying the definition of the previous actions. Although with different meanings, we will use the same notations in this new context.  

Let $s,u\in \ODP_n\setminus\{1\}$ and 
suppose that $\dom(us)=\{i_1,\ldots,i_k\}$, for some $1\le i_1<\cdots<i_k\le n$ and $0\le k<n$. Define the elements $u\la s,u^s\in\POI_n$ by 
$$
\dom(u\la s)=\dom(us) \quad\text{and}\quad 
\im(u\la s)=\{1,1+i_2-i_1,\ldots,1+i_k-i_1\},
$$
$$
\dom(u^s)=\{1,1+i_2us-i_1us,\ldots,1+i_kus-i_1us\} \quad\text{and}\quad \im(u^s)=\im(us)
$$
(considering $u\la s=u^s=\emptyset$ if $us=\emptyset$). 

Notice that, clearly,  
$$ 
u\la s\in\ODP_n^- \quad\text{and}\quad u^s\in\ODP_n^+. 
$$
Moreover 
$$
\im(u\la s)=\dom(u^s). 
$$

Define also 
$$
1\la s=s,\quad u^1=u,\quad u\la 1=1,\quad 1^s=1,\quad 1\la 1=1 
\quad\text{and}\quad 1^1=1. 
$$

As for the first studied case, it is easy to check the equality 
\begin{equation}\label{comutealso}
 (u\la s)u^s=us,
\end{equation}
for all $u\in \ODP_n^+$ and $s \in \ODP_n^-$,  
and we may consider the following two functions: 
$$
\begin{array}{ccccccc}
\delta:&\ODP_n^+ &\longrightarrow&\mathcal{T}(\ODP_n^-)&&\\
&u&\longmapsto&\delta_u:\ODP_n^-&\longrightarrow&\ODP_n^-\\
&&&s&\longmapsto&u\la s
\end{array}
$$
and 
$$
\begin{array}{ccccccc}
\varphi:&\ODP_n^-&\longrightarrow&\mathcal{T}(\ODP_n^+)&&\\
&s&\longmapsto&\varphi_s:\ODP_n^+&\longrightarrow&\ODP_n^+\\
&&&u&\longmapsto&u^s.
\end{array}
$$

By exact replication of the proofs of Lemma \ref{morphis} and Lemma \ref{rules}, we may prove the following lemma: 

\begin{lemma}\label{newactions} 
The functions $\delta$ and $\varphi$ are a monoidal left action of $\ODP_n^+$ on $\ODP_n^-$ and a monoidal right action of $\ODP_n^-$ on $\ODP_n^+$, respectively. 
\end{lemma} 

This lemma allows us to consider the bilateral semidirect product $\ODP_n^-\kunze\ODP_n^+$ associated with $\delta$ and $\varphi$, which is, likewise $\POI_n^- \kunze \POI_n^+$, a non regular and non idempotent commuting aperiodic monoid. We may also consider the function 
$$
\begin{array}{lccc}
\mu:&\ODP_n^-\kunze\ODP_n^+&\longrightarrow&\ODP_n\\
&(s,u)&\longmapsto&su~, 
\end{array}
$$
which is, by (\ref{comutealso}), clearly a homomorphism. Moreover, 
let $t\in\ODP_n$ be such that $\dom(t)=\{i_1,\ldots,i_k\}$, 
for some $1\le i_1<\cdots<i_k\le n$ and $0\le k\le n$, and  
define elements $s\in\ODP_n^-$ and $u\in\ODP_n^+$ by 
$$
\dom(s)=\dom(t),
$$
$$
\im(s)=\{1,1+i_2-i_1,\ldots,1+i_k-i_1\}=\{1,1+i_2t-i_1t,\ldots,1+i_kt-i_1t\}=\dom(u)
$$
and
$$
\im(u)=\im(t).  
$$ 
Then $t=su=(s,u)\mu$ and so $\mu$ is onto homomorphism.  

Hence, we have the following result, with which we finish this section.  

\begin{theorem}
The monoid $\ODP_n$ is a homomorphic image of $\ODP_n^-\kunze\ODP_n^+$.
\end{theorem}

\section{On the monoids $\PODI_n$ and $\DP_n$}\label{unilateral}

Let 
$$
h=\left(\begin{array}{ccccc}
1&2&\cdots&n-1&n \\ 
n&n-1&\cdots&2&1
\end{array}\right).
$$
Then $h\in\DP_n$ (and so $h\in\PODI_n$). Moreover, the identity (on $X_n$) and $h$ are the only permutations of $\PODI_n$ (and so of $\DP_n$). 
On the other hand, given $\alpha\in \PODI_n$, it is clear that $\alpha$ is an order-reversing transformation if and only if $h\alpha$ (and $\alpha h$) is an order-preserving transformation (see \cite{Fernandes&Gomes&Jesus:2004}). Hence, as $\alpha=h^2\alpha=h(h\alpha)$, it follows that the monoids $\PODI_n$ and $\DP_n$ are generated by $\POI_n\cup \{h\}$ and $\ODP_n\cup \{h\}$, respectively. 
Furthermore, we may see the cyclic group of order two $\mathcal{C}_2=\{1,h\}$ as a submonoid of both the monoids $\PODI_n$ and $\DP_n$. 
Notice that, given $x,y\in{\cal{C}}_2$, we have $xy=yx$ and $x^2=y^2=1$.

\smallskip 

First, we turn our attention to the monoid $\PODI_n$. We obtain a semidirect decomposition of it in terms of its submonoids $\POI_n$ and $\mathcal{C}_2$. 

For each $x\in\mathcal{C}_2$ and $s\in\POI_n$, define the element $x\la s=xsx\in\POI_n$.  
Then, consider the function 
$$
\begin{array}{ccccccc}
\delta:&\mathcal{C}_2 &\longrightarrow&\mathcal{T}(\POI_n)&&\\
&x&\longmapsto&\delta_x:\POI_n&\longrightarrow&\POI_n\\
&&&s&\longmapsto&x\la s~. 
\end{array}
$$
Since $(xy)\la s=xysxy=xysyx=x\la (ysy)=x\la (y\la s)$ and $1\la s=s$, 
for $x,y\in \mathcal{C}_2$ and $s\in \POI_n$, then $\delta$ is an
anti-homomorphism of monoids. On the other hand, 
for $x\in \mathcal{C}_2$ and $s,r\in \POI_n$, 
we have $x\la(sr)=xsrx=xs1rx=xsxxrx=(x\la s)(x\la r)$ and $x\la 1=x1x=x^2=1$. Thus $\delta$ induces a semidirect product $\POI_n\se \mathcal{C}_2$. 

It is easy to prove that $\POI_n\se \mathcal{C}_2$ is an inverse monoid. In fact, it is a routine matter to check that the idempotents of $\POI_n\se \mathcal{C}_2$ commute 
(the idempotents of $\POI_n\se \mathcal{C}_2$ are of the form $(e,1)$, 
with $e$ an idempotent of $\POI_n$) and, given $(s,x)\in \POI_n\se \mathcal{C}_2$, 
the element $(xs^{-1}x,x)$ of $\POI_n\se \mathcal{C}_2$ is an (and so \textit{the}) inverse  of $(s,x)$. 
Moreover, we have:

\begin{theorem}\label{ps1PODIn}
The monoid $\PODI_n$ is a homomorphic image of $\POI_n\se \mathcal{C}_2$.
\end{theorem}
\begin{proof}
Consider the function  
$$
\begin{array}{lccc}
\mu:&\POI_n\se \mathcal{C}_2 &\longrightarrow&\PODI_n\\
&(s,x)&\longmapsto&sx~. 
\end{array}
$$ 
Then, for $s,r\in \POI_n$ and $x,y\in \mathcal{C}_2$, we have 
$$
((s,x)(r,y))\mu=(s (x\la r),xy)\mu=(sxrx,xy)\mu=sxrx^2y=sxry=(s,x)\mu(r,y)\mu. 
$$ 
Thus $\mu$ is a homomorphism. 
On the other hand, let $t \in \PODI_n$. If $t\in\POI_n$ then $t=t1=(t,1)\mu$, otherwise  
$th\in\POI_n$ and $t=(th)h=(th,h)\mu$. Hence $\mu$ is surjective. 
\end{proof}  

Observe that, clearly, $\mu$ also separates idempotents,  
i.e. the restriction of $\mu$ to the set of the idempotents of $\POI_n\se \mathcal{C}_2$ is an injective function. 

The next result follows immediately from Theorem \ref{ps1PODIn}. 

\begin{corollary}\label{pv1PODIn}
$\pv{PODI}\subseteq\pv{POI}\se\pv{Ab_2}$. 
\end{corollary}

On the other hand, we also have: 

\begin{lemma}\label{pv2PODIn}
$\POI_n\se\mathcal{C}_2\in\pv{PODI}$.
\end{lemma}
\begin{proof}
It is easy to show that the function 
$$
\begin{array}{ccc}
\POI_n\se \mathcal{C}_2 &\longrightarrow&\PODI_n\times\mathcal{C}_2\\
(s,x)&\longmapsto&(sx,x)
\end{array}
$$ 
is an injective homomorphism.
\end{proof}

Supported by this result, we formulate the following conjecture: 

\begin{conjecture}
 $\pv{PODI}=\pv{POI}\se\pv{Ab_2}$.
\end{conjecture}

\smallskip 

Notice that, since $\mathcal{C}_2$ is a commutative monoid, the left action of $\mathcal{C}_2$ on $\POI_n$ may also be considered as a right action. Furthermore, similar results to Theorem \ref{ps1PODIn} and Corollary \ref{pv1PODIn} (and Lemma \ref{pv2PODIn}) also hold for reverse semidirect products.  

\medskip 

We finish this section by establishing the analogous result to Theorem \ref{ps1PODIn} for the monoid $\DP_n$. This aim will be accomplish by noticing that $\DP_n$ is a submonoid of $\PODI_n$ that fits in the general framework described below. 

\smallskip 

Let $S$ be a monoid and let $S_1$ and $S_2$ be two submonoids of $S$. 
Let $\delta$ be a left action of $S_2$ on $S_1$
such that the function  
$$
\begin{array}{cccc}
\mu:&S_1\se S_2&\longrightarrow&S\\
&(s,u)&\mapsto&su
\end{array}
$$
is a homomorphism. 
Let $T$ be a submonoid of $S$, $T_1$ a submonoid of $S_1$ and $T_2$ a submonoid of $S_2$. 
It is a routine matter to check that, 
if $(s)(u)\delta \in T_1$, for all $s\in T_1$ and $u\in T_2$, then $\delta$ induces a (restriction) 
left action of $T_2$ on $T_1$ and the corresponding semidirect product $T_1\se T_2$ is a submonoid of $S_1\se S_2$. If, in addition, 
$T=T_1T_2$ then 
$$
\begin{array}{cccc}
\mu{|_{T_1\rtimes T_2}}:&T_1\se T_2&\longrightarrow&T\\
&(s,u)&\mapsto&su
\end{array}
$$
is a surjective homomorphism.

\smallskip 

For $s\in \ODP_n$ and $x\in \mathcal{C}_2$, it is clear that $x\la s=xsx \in \ODP_n$. Thus, we may consider the semidirect product $\ODP_n\se \mathcal{C}_2$ induced by the left action $\delta$ of 
$\mathcal{C}_2$ on $\POI_n$. Moreover, since $\DP_n=\ODP_n\mathcal{C}_2$, then 
$\mu{|_{\ODP_n\rtimes\mathcal{C}_2}}:\ODP_n\rtimes\mathcal{C}_2\longrightarrow\DP_n$ 
is a surjective homomorphism and so we have: 

\begin{theorem}\label{psDPn}
The monoid $\DP_n$ is a homomorphic image of $\ODP_n\se \mathcal{C}_2$.
\end{theorem}


\lastpage 

\end{document}